\documentclass[11pt,oneside,reqno]{amsart}

\usepackage{amssymb}
\usepackage{algpseudocode}
\usepackage{algorithm}
\usepackage{subcaption}
\usepackage[singlelinecheck=false]{caption}
\usepackage{verbatim}
\usepackage{graphicx}
\usepackage{bbm}
\usepackage[T1]{fontenc}
\usepackage[noadjust]{cite}
\RequirePackage{dsfont} 
 \scrollmode
\allowdisplaybreaks
\usepackage{graphicx}
\usepackage{epstopdf}
\usepackage[usenames]{color}
\definecolor{darkred}{RGB}{139,0,0}
\definecolor{darkgreen}{RGB}{0,100,0}
\definecolor{darkmagenta}{RGB}{139,0,139}

\usepackage{amsmath}
\usepackage{tikz}

\makeatletter
\newcommand{\xleftrightarrow}[2][]{\ext@arrow 3359\leftrightarrowfill@{#1}{#2}}
\makeatother

\newcommand{\xdasharrow}[2][->]{
\tikz[baseline=-\the\dimexpr\fontdimen22\textfont2\relax]{
\node[anchor=south,font=\scriptsize, inner ysep=1.5pt,outer xsep=2.2pt](x){#2};
\draw[shorten <=3.4pt,shorten >=3.4pt,dashed,#1](x.south west)--(x.south east);
}
}

\newcommand{\DEBUG}{}

\ifdefined\DEBUG%

  \def\rem#1{{\marginpar{\raggedright\scriptsize #1}}}
  \newcommand{\pmr}[1]{\rem{\color{blue}{$\bullet$ #1}}}
  \newcommand{\ppr}[1]{\rem{\color{red}{$\bullet$ #1}}}
 \else%

  \newcommand{\ppr}[1]{}
  \newcommand{\pmr}[1]{}
 \fi

\theoremstyle{plain}
\newtheorem{theorem}{Theorem}
\newtheorem{lemma}{Lemma}

\theoremstyle{definition}
\newtheorem{remark}{Remark}

\usepackage{enumitem}
\begin{document}

\title
[Stability of randomized Taylor schemes]
{A note on the probabilistic stability \\ of randomized Taylor schemes}

\author[T. Bochacik]{Tomasz Bochacik}
\address{AGH University of Science and Technology,
Faculty of Applied Mathematics,\newline
Al. A.~Mickiewicza 30, 30-059 Krak\'ow, Poland}
\email{bochacik@agh.edu.pl}

\begin{abstract}
We study the stability of randomized Taylor schemes for ODEs. We consider three notions of probabilistic stability: asymptotic stability, mean-square stability, and stability in probability. We prove fundamental properties of the probabilistic stability regions and benchmark them against the absolute stability regions for deterministic Taylor schemes. 
\newline
\newline
\textbf{Key words:} randomized Taylor schemes, mean-square stability, asymptotic stability, stability in probability
\newline
\newline
\textbf{MSC 2010:} 65C05,\ 65L05,\ 65L20
\end{abstract}

\maketitle
\tableofcontents
%%%%%%%%%%%%%%%%%%%%%%%%%%%%%%%%%%%%%%%%%%%%%%%%%%%%%%%%%%%%%%%%%%%%
%%%%%%%%%%%%%%%%%%%%%%%%%%%%%%
%%%%%%%%%%%%%%%%%%%%%%%%%%%%%%

\section{Introduction}

The study of randomized algorithms approximating the solutions of initial value problems for ODEs dates back to early 1990s, cf. \cite{stengle1, stengle2}. 

So far, the main focus has been on convergence of randomized algorithms, see for example \cite{backward_euler, JenNeuen, KruseWu_1}. Randomized algorithms tend to converge faster than their deterministic counterparts, especially for the problems of low regularity. Error bounds are usually established using certain martingale inequalities and classical tools such as Gronwall's inequality. In many papers, error analysis was combined with the discussion of algorithms' optimality (in the Information Based Complexity sense), cf. \cite{daun1, daun2, gocwin, HeinMilla, Kac1, Kac2, Kac3}, also in the setting of inexact information, cf. \cite{randRK,randEuler}. Randomized Taylor schemes, which will be of particular interest in this paper, were shown in \cite{HeinMilla} to achieve the optimal rate of convergence under mild regularity conditions. 

Other aspects of randomized algorithms, such as stability, have been largely omitted. The aim of this paper is to make a step towards filling this gap.

The stability of deterministic algorithms for ODEs has been comprehensively studied in the literature, see, for example, \cite{hairer}. In the stability analysis, we consider a test problem which is simple enough but retains features present in a wider class of problems. Then, we investigate for which choices of the step-size, the method reproduces the characteristics of the test equation, cf. \cite{higham2000}. In the context of ODEs, we usually take a linear, scalar, and autonomous test problem. 

The same test problem is used for stability analysis of randomized algorithms for ODEs. However, the approximated solution generated by a randomized method is random. Hence, analysis of its behaviour in infinity depends on the type of convergence. This naturally leads to notions of the mean-square stability and the asymptotic (almost-sure) stability, which have been previously considered in \cite{higham2000, mitsui} in the context of stochastic differential equations. This framework, enriched with the notion of stability in probability, was used in \cite{randRK, randEuler} to characterize stability regions of randomized Euler schemes and the randomized two-stage Runge-Kutta scheme.

This paper, according to our best knowledge, is the first attempt to apply the concept of probabilistic stability to higher-order randomized methods for ODEs, namely to the family of randomized Taylor schemes defined in \cite{HeinMilla}. Since these methods do not involve implicitness, they will not be A-stable. However, we can characterize probabilistic stability regions of  these methods in a quite detailed way. We establish their basic properties such as openness, boundedness, symmetry. Moreover, we study inclusions between them and compare them to the reference sets corresponding to deterministic methods. Finally, we provide counterexamples for some hypothetical properties which do not hold for probabilistic stability regions of randomized Taylor schemes in general (that is, for the method of any order).

In Section \ref{sec:prel} we give basic definitions and introduce the notation. In particular, we recall the definitions of the family of randomized Taylor schemes and of the probabilistic stability regions. In Section \ref{sec:main}, which is the main part of this paper, we characterize probabilistic stability regions for randomized Taylor schemes. Conclusions are discussed in Section \ref{sec:concl}. Finally, in Appendix \ref{sec:appendix} we prove some technical lemmas.

%%%%%%%%%%%%%%%%%%
\section{Preliminaries} \label{sec:prel}

\subsection{The family of randomized Taylor schemes}
We deal with initial value problems of the following form: 
\begin{equation}
	\label{eq:ode}
		\left\{ \begin{array}{ll}
			z'(t)= f(t,z(t)), \ t\in [a,b], \\
			z(a) = \eta, 
		\end{array}\right.
\end{equation}
where $-\infty < a < b < \infty$, $\eta\in\mathbb{R}^d$, $f\colon [a,b]\times\mathbb{R}^d\to\mathbb{R}^d$, $d\in\mathbb{Z}_+$. 

We use the definition of the family of randomized Taylor schemes given in \cite{HeinMilla}. We fix $n\in\mathbb{N}$, $n\geq 2$ and put $h=\frac{b-a}{n}$, $t_j = a+jh$ for $j\in\{0,1,\ldots,n\}$, $\theta_j = t_{j-1} + \tau_jh$, $\tau_j\sim U(0,1)$ for $j\in\{1,\ldots,n\}$. We assume that the family of random variables $\{\tau_1,\ldots,\tau_n\}$ is independent.

Let $r\in\mathbb{N}$ (we assume that $0\in\mathbb{N}$). We set $v^r_0=\eta$. If $k\in\{1,\ldots,n\}$ and $v^r_{k-1}$ is already defined, we consider the following local problem:
\begin{equation}
	\label{eq:local}
		\left\{ \begin{array}{ll}
			(u^r_k)'(t)= f(t,u^r_k(t)), \ t\in [t_{k-1},t_{k}], \\
			u^r_k(t_{k-1}) = v^r_{k-1} .
		\end{array}\right.
\end{equation}
We define
\begin{equation}
	\label{eq:polynomial}
		p^r_k(t) = \sum_{j=0}^{r+1} \frac{(u^r_k)^{(j)}(t_{k-1})}{j!} (t-t_{k-1})^j \mathbbm{1}_{[t_{k-1},t_k]}(t)
\end{equation}
and 
\begin{equation}
	\label{eq:taylor}
		v^r_k = p^r_k(t_k) + h\cdot \bigl( f(\theta_k, p^r_k(\theta_k)) - (p^r_k)'(\theta_k) \bigr).
\end{equation}
Note that for $t\in (t_{k-1},t_k)$, where $k\in\{1,\ldots,n\}$, we have
\begin{equation}
	\label{eq:polynomial-diff}
		(p^r_k)'(t) = \sum_{j=0}^{r+1} \frac{(u^r_k)^{(j)}(t_{k-1})}{(j-1)!} (t-t_{k-1})^{j-1}.
\end{equation}

The algorithm returns the sequence $\bigl( v_k^r \bigr)_{k=0}^n$, which approximates values of the exact solution $z$ of \eqref{eq:ode} at points $t_0,\ldots,t_n$. If we neglect the second term in the right-hand side of \eqref{eq:taylor}, we get the classical deterministic Taylor scheme.

Note that for $r=0$ we obtain the definition of randomized two-stage Runge-Kutta scheme given in \cite{randRK, KruseWu_1}. Stability analysis for this algorithm was performed in \cite{randRK}.

\subsection{Probabilistic stability of randomized Taylor schemes}
Let us consider the classical test problem
\begin{equation}
	\label{TEST_PROBLEM}
		\left\{ \begin{array}{ll}
			z'(t)= \lambda z(t), \ t\geq 0, \\
			z(0) = \eta
		\end{array}\right.
\end{equation}
with $\lambda\in\mathbb{C}$ and $\eta\neq 0$. The exact solution of \eqref{TEST_PROBLEM} is $z(t)=\eta\exp(\lambda t)$. We note that 
\begin{equation*}
    \lim\limits_{t\to\infty} z(t)=0 \ \hbox{iff} \ \Re(\lambda)<0.
\end{equation*}

For a fixed step-size $h>0$, we apply the scheme \eqref{eq:taylor} with the mesh $t_k=kh$, $k\in\mathbb{N}$, to the test problem \eqref{TEST_PROBLEM}. The local problem \eqref{eq:local} takes on the following form:
\begin{equation}
	\label{eq:local2}
		\left\{ \begin{array}{ll}
			(u^r_k)'(t)= \lambda u^r_k(t), \ t\in [t_{k-1},t_{k}], \\
			u^r_k(t_{k-1}) = v^r_{k-1} .
		\end{array}\right.
\end{equation}
As a result, we obtain a sequence $(v^r_k)_{k=0}^\infty$, which approximates the values of $z$ at $t_k$ ($k\in\mathbb{N}$) and whose values are given by
\begin{equation}
	\label{eq:v}
        v^r_k = \eta \cdot \prod_{l=1}^k f_{r,\lambda h} (\tau_l),
\end{equation}
where $\tau_1,\tau_2,\ldots$ are independent random variables with uniform distribution on $[0,1]$, and
\begin{equation} \label{eq:f}
    f_{r,z} \colon \mathbb{R} \ni t \mapsto \sum_{j=0}^{r+1} \frac{z^j}{j!} + \frac{t^{r+1}}{(r+1)!} z^{r+2} \in \mathbb{C}.
\end{equation}

In fact, using \eqref{eq:local2} and proceeding by induction with respect to $j$, we get 
\begin{equation*}
        (u^r_k)^{(j)}(t_{k-1}) = \lambda^j v^r_{k-1}.
\end{equation*}
Thus, by taking $t=t_k$ and $t=\theta_k$ in \eqref{eq:polynomial} and \eqref{eq:polynomial-diff}, we get
\begin{align*}
        p^r_k(t_k) & = \sum_{j=0}^{r+1} \frac{(\lambda h)^j}{j!} v^r_{k-1},  \\ p^r_k(\theta_k) & = \sum_{j=0}^{r+1} \frac{(\lambda h \tau_k)^j}{j!} v^r_{k-1},  \\ (p^r_k)'(\theta_k) & = \sum_{j=1}^{r+1} \frac{\lambda^{j}(h\tau_k)^{j-1}}{(j-1)!} v^r_{k-1} = \lambda\sum_{j=0}^{r} \frac{(\lambda h\tau_k)^{j}}{j!} v^r_{k-1}. 
\end{align*}
By \eqref{eq:taylor} and the above three lines, we obtain the following recurrence:
\begin{align*}
 v^r_k & = p^r_k(t_k) + \lambda h p^r_k(\theta_k) - h (p^r_k)'(\theta_k) = \Bigl( \sum_{j=0}^{r+1} \frac{(\lambda h)^j}{j!} + \lambda h \cdot \frac{(\lambda h \tau_k)^{r+1}}{(r+1)!} \Bigr) \cdot v^r_{k-1},
\end{align*}
which leads to \eqref{eq:v}.

Similarly as in \cite{randRK}, we consider three sets
\begin{eqnarray}
    &&\mathcal{R}^r_{MS} = \{\lambda h\in\mathbb{C} \colon v^r_k \to 0 \ \hbox{in $L^2(\Omega)$ as} \ k\to \infty \}, \label{eq:regMS} \\
    &&\mathcal{R}^r_{AS} = \{\lambda h\in\mathbb{C} \colon v^r_k \to 0 \  \hbox{almost surely as} \ k\to \infty\},\label{eq:regAS} \\
     &&\mathcal{R}^r_{SP} = \left\{ \lambda h\in \mathbb{C} \colon  v^r_k \to 0 \text{ in probability as} \ k\to \infty  \right\},
    \label{eq:regSP}
\end{eqnarray}
where we call $\mathcal{R}^r_{MS}$ the region of mean-square stability,  $\mathcal{R}^r_{AS}$ -- the region of asymptotic stability, and $\mathcal{R}^r_{SP}$ -- the region of stability in probability. 

\begin{remark}
If $h\lambda \in \mathcal{R}^r_{AS}$, the approximated solution of the test problem converges to $0$ for virtually every fixing of $\tau_1,\tau_2,\ldots$ from the interval $[0,1]$. That is, for almost every $\omega\in\Omega$ and for every $\varepsilon>0$, there exists $k_0=k_0(\omega, \varepsilon)$ such that $|v_k^r(\omega)|<\varepsilon$ for every $k\geq k_0$.  

The mean-square stability implies that for each pre-specfied threshold $\varepsilon>0$ and each pre-specified probability level $\delta\in (0,1)$, the point-wise approximations $v^r_k$ (for sufficiently big $k$) are bounded by $\varepsilon$ with probability at least $1-\delta$. In fact, let us assume that $h\lambda\in\mathcal{R}^r_{MS}$ and let us set $\varepsilon>0$ and $\delta\in (0,1)$. Then there exists $k'_0=k'_0(\varepsilon,\delta)\in\mathbb{Z}_+$ such that for each $k\geq k'_0$, $\|v^r_k\|_{L^2(\Omega)}<\sqrt{\delta}\varepsilon$. By Markov's inequality,
$$\mathbb{P}\bigl( |v^r_k| \geq \varepsilon \bigr) \leq \frac{\|v^r_k\|^2_{L^2(\Omega)}}{\varepsilon^2} <\delta.$$

Hence, the asymptotic stability guarantees that for each specific run of the algorithm, the approximated solution will converge to the exact solution in infinity (provided that $h\lambda \in \mathcal{R}^r_{AS}$). However, the pace of convergence may significantly vary for different fixings of $\tau_1,\tau_2,\ldots$. On the other hand, the mean-square stability provides an insight on whether we can expect consistent asymptotic behaviour in many independent runs of the algorithm. 
\end{remark}

\section{Main results} \label{sec:main}

For fixed $r\in\mathbb{N}$, let
\begin{align} 
    F_r\colon & \mathbb{C} \ni z \mapsto \mathbb{E} |f_{r,z}(\tau)|^2 = \int\limits_0^1 |f_{r,z}(t)|^2\,\mathrm{d}t \label{eq:F} \\ & \ \ \ \ \ \ \ \ \ \ \ \ \ \ \ \ \ \ \ \ \ \ \ \ \ \ \ \ \ \ \ \ \ \   = \int\limits_0^1 \Bigl|\sum_{j=0}^{r+1} \frac{z^j}{j!} + \frac{t^{r+1}}{(r+1)!} z^{r+2}\Bigr|\,\mathrm{d}t \in [0,\infty),  \notag \\
    G_r\colon & \mathbb{C} \ni z \mapsto \mathbb{E}\bigl( \ln |f_{r,z}(\tau)|\bigr) = \int\limits_0^1 \ln |f_{r,z}(t)|\,\mathrm{d}t \in \mathbb{R}, \label{eq:G}
\end{align}
where $\tau\sim U([0,1])$ and $f_{r,z}$ is given by \eqref{eq:f}. From \eqref{eq:v}, \eqref{eq:regMS} and \eqref{eq:F} we obtain
\begin{equation} \label{eq:MS}
    \mathcal{R}_{MS}^r = \bigl\{ z\in\mathbb{C} \colon F_r(z) < 1 \bigr\} = \left\{ z\in\mathbb{C} \colon \int\limits_0^1 \Bigl|\sum_{j=0}^{r+1} \frac{z^j}{j!} + \frac{t^{r+1}}{(r+1)!} z^{r+2}\Bigr|\,\mathrm{d}t < 1 \right\}.
\end{equation}

For further analysis of the probabilistic regions of randomized Taylor schemes, we need Lemmas \ref{lemma:R_AS}--\ref{A-cont}. They are formulated and proven in Appendix \ref{sec:appendix}.

By Lemma \ref{lemma:R_AS} and Corollary 1 in \cite{randRK}, regions $\mathcal{R}^r_{AS}$ and $\mathcal{R}^r_{SP}$ are equal and can be expressed as
\begin{equation} \label{eq:RAS-SP}
    \mathcal{R}^r_{AS} = \mathcal{R}^r_{SP} = \bigl\{ z\in\mathbb{C} \colon G_r(z) < 0 \bigr\}.
\end{equation} 

Apart from $\mathcal{R}_{MS}^r, \mathcal{R}_{AS}^r$ and $\mathcal{R}_{SP}^r$, we consider the following reference set:
\begin{equation} \label{eq:ref}
    \mathcal{R}_{ref}^r =\Bigl\{ z\in\mathbb{C} \colon \eta\prod_{l=1}^k \mathbb{E} f_{r,z}(\tau_l) \to 0 \text{ as } k\to\infty \Bigr\} = \Bigl\{ z\in\mathbb{C} \colon \Bigl| \sum_{j=0}^{r+2} \frac{z^j}{j!} \Bigr| < 1 \Bigr\}.
\end{equation}
Note that the reference set $\mathcal{R}_{ref}^r$ is the absolute stability region for the deterministic Taylor scheme of order $r+2$. Hence, our analysis is consistent with \cite{randRK}, where probabilistic stability regions of the randomized two-stage Runge-Kutta scheme (i.e., the randomized Taylor scheme with $r=0$) were benchmarked against the absolute stability region of the mid-point method (i.e., the deterministic Taylor scheme with $r=1$).

Now we are ready to establish the main results of this paper -- Theorem \ref{theorem_MS_1} and Theorem \ref{theorem_MS_2}. They extend Theorem 3 and Theorem 4 from \cite{randRK}, which covered only the case of $r=0$, onto the case of any $r\in\mathbb{N}$. However, as of now we have not managed to generalize the results from \cite{randRK} related to stability intervals.

\begin{theorem}
\label{theorem_MS_1}
For each $r\in\mathbb{N}$, the sets $\mathcal{R}^r_{MS}$, $\mathcal{R}^r_{AS}$ and $\mathcal{R}^r_{ref}$ are open and symmetric with respect to the real axis.
\end{theorem}
\begin{proof}
Sets $\mathcal{R}^r_{MS} = F_r^{-1}((-\infty,1))$ and $\mathcal{R}^r_{AS} = G_r^{-1}((-\infty,0))$ are open due to Lemma \ref{lemma_F} and Lemma \ref{A-cont}, respectively. Since the function 
\begin{equation*}
\hat F_r \colon \mathbb{C} \ni z \mapsto \Bigl| \sum_{j=0}^{r+2} \frac{z^j}{j!} \Bigr| \in [0,\infty)
\end{equation*}
is continuous, set $\mathcal{R}^r_{ref} = \hat F_r^{-1} ((-\infty,1))$ is open as well. 

Note that $|f_{r,z}(t)| = |f_{r,\bar z}(t)|$ for all $z\in\mathbb{C}$ and $t\in [0,1]$, which implies that $F_r(z) = F_r(\bar z)$ and $G_r(z)=G_r(\bar z)$ for all $z\in\mathbb{C}$, cf. \eqref{eq:f}, \eqref{eq:F} and \eqref{eq:G}. This combined with \eqref{eq:MS} and \eqref{eq:RAS-SP} immediately gives the symmetry of $\mathcal{R}^r_{MS}$ and $\mathcal{R}^r_{AS}$ with respect to the real axis. The same property for $\mathcal{R}^r_{ref}$ follows from the fact that 
$$ \Bigl| \sum_{j=0}^{r+2} \frac{z^j}{j!} \Bigr| = \Bigl| \sum_{j=0}^{r+2} \frac{\bar z^j}{j!} \Bigr|$$
for all $z\in\mathbb{C}$.
\end{proof}

The following Theorem \ref{theorem_MS_2} shows that the mean-square stability is a stronger property than the asymptotic stability. Furthermore, probabilistic stability regions of the randomized Taylor scheme for any $r\in\mathbb{N}$ are bounded. This implies that none of randomized Taylor schemes is A-stable in any of the considered probabilistic senses (i.e., the left complex half-plane is not contained in any of the considered stability regions). However, the left half-plane is contained in the sum over $r$ of stability regions. 

\begin{theorem}
\label{theorem_MS_2}
For each $r\in\mathbb{N}$, there exists $\gamma_r\in (0,\infty)$ such that
\begin{equation} \label{eq:MSbound}
    \mathcal{R}^r_{MS}\subset\mathcal{R}^r_{ref}\cap\mathcal{R}^r_{AS}\subset\mathcal{R}^r_{ref}\cup\mathcal{R}^r_{AS}\subset \{z\in \mathbb{C} \colon |z|<\gamma_r\}.
\end{equation} 
Moreover, 
\begin{equation} \label{eq:C-}
    \mathbb{C}_{-} \subset \bigcup_{r=0}^\infty \mathcal{R}^r_{MS}.
\end{equation}
\end{theorem}

\begin{proof}
Inclusion $\mathcal{R}_{MS}^r\subset\mathcal{R}_{ref}^r$ for all $r\in\mathbb{N}$ follows from the following inequality:
$$ F_r(z) = \mathbb{E} \bigl| f_{r,z}(\tau) \bigr|^2 \geq  \bigl| \mathbb{E} f_{r,z}(\tau) \bigr|^2 = \Bigl| \sum_{j=0}^{r+1} \frac{z^j}{j!} + \frac{\mathbb{E}\tau^{r+1}}{(r+1)!} z^{r+2} \Bigr|^2 = \Bigl| \sum_{j=0}^{r+2} \frac{z^j}{j!} \Bigr|^2.$$
We used the fact that $\mathbb{E}|Z|^2 \geq |\mathbb{E}Z|^2$ for any complex random variable $Z$. Since convergence in $L^2(\Omega)$ implies convergence in probability, we have $\mathcal{R}_{MS}^r\subset\mathcal{R}_{SP}^r = \mathcal{R}_{AS}^r$ for all $r\in\mathbb{N}$, cf. \eqref{eq:regMS}, \eqref{eq:regSP} and \eqref{eq:RAS-SP}. As a result, we obtain the first inclusion in \eqref{eq:MSbound}.

Region $\mathcal{R}_{ref}^r$ is bounded for each $r\in\mathbb{N}$ because
$$\Bigl| \sum_{j=0}^{r+2} \frac{z^j}{j!} \Bigr| \geq \frac{|z|^{r+2}}{(r+2)!} - \sum_{j=0}^{r+1} \frac{|z|^j}{j!}$$
and the right-hand side of the above inequality tends to infinity when $|z|\to\infty$. Thus, there exists $\gamma^1_r>0$ such that $\Bigl| \sum_{j=0}^{r+2} \frac{z^j}{j!} \Bigr|\geq 1$ for all $z\in\mathbb{C}$ with $|z| \geq \gamma^1_r$. 

To show that $\mathcal{R}_{AS}^r$ is bounded, let us express $G_r(z)$ for $z\in\mathbb{C}\setminus\{0\}$ in the same fashion as in the proof of Lemma \ref{A-cont}:
\begin{equation} \label{eq:RASbound}
    G_r(z) = \ln\Bigl(\frac{|z|^{r+2}}{(r+1)!} \Bigr) + H_r(h_r(z)) \geq (r+2) \ln |z| - \ln ((r+1)!) + \inf_{z\in\mathbb{C}} H_r(z),
\end{equation}
cf. \eqref{eq:G2}, \eqref{eq:h} and \eqref{eq:H}. Note that $\displaystyle \inf_{z\in\mathbb{C}} H_r(z)$ is finite because 
$$ H_r(z) \geq  \int\limits_0^1 \ln \bigl||z| - t^{r+1}\bigr|\,\mathrm{d}t\geq  \int\limits_0^1 \ln (2 - t^{r+1})\,\mathrm{d}t\geq 0 $$
for $|z|>2$ and 
$\displaystyle \inf_{|z|\leq 2} H_r(z)$ is finite by Lemma \ref{A-cont} and the Weierstrass extreme value theorem. The right-hand side of \eqref{eq:RASbound} tends to infinity when $|z|\to\infty$. Hence, there exists $\gamma^2_r>0$ such that $G_r(z)\geq 0$ for all $z\in\mathbb{C}$ with $|z|\geq \gamma^2_r$. Taking $\gamma_r = \max\{\gamma^1_r,\gamma^2_r\}$ leads to the third inclusion in \eqref{eq:MSbound}.

To see \eqref{eq:C-}, let us consider $z\in\mathbb{C}$ such that $z\notin \bigcup_{r=0}^\infty \mathcal{R}^r_{MS}$. Then $F_r(z)\geq 1$ for all $r\in\mathbb{N}$ and as a result
\begin{equation} \label{eq:limsup}
   1 \leq \limsup_{r\to\infty} \int\limits_0^1 \bigl| f_{r,z}(t)\bigr|^2\,\mathrm{d}t \leq \int\limits_0^1 \limsup_{r\to\infty} \bigl| f_{r,z}(t)\bigr|^2\,\mathrm{d}t = \int\limits_0^1  | e^z |^2\,\mathrm{d}t = e^{2\Re(z)}. 
\end{equation}
Hence, $\Re(z)\geq 0$ and \eqref{eq:C-} follows. The second inequality in \eqref{eq:limsup} is based on Fatou's lemma. Note that $f_{r,z}$ is continuous in $\mathbb{R}$, which guarantees that the function $[0,1]\ni t \mapsto \bigl| f_{r,z}(t)\bigr|^2 \in [0,\infty)$ is continuous as well and thus Borel measurable. In the third passage in \eqref{eq:limsup} we use the fact that $\displaystyle \lim_{r\to\infty} f_{r,z}(t)$ exists and is equal to $e^z$ for each $t\in[0,1]$. 
\end{proof}

In Figure \ref{fig:regions-ref-ms}, we plot $\mathcal{R}^{r}_{ref}$, $\mathcal{R}^{r}_{MS}$, and $\mathcal{R}^{r}_{AS}$ for $r\in\{0,1,2,3,4\}$. Based on these plots and some calculations, we have rejected several hypotheses about potential properties of these regions. Counterexamples are provided in Remarks \ref{remark:inclusions}--\ref{remark:connect}.

\begin{figure}
    \centering
    \begin{subfigure}{0.45\textwidth}
    \centering
        \includegraphics[width=\linewidth]{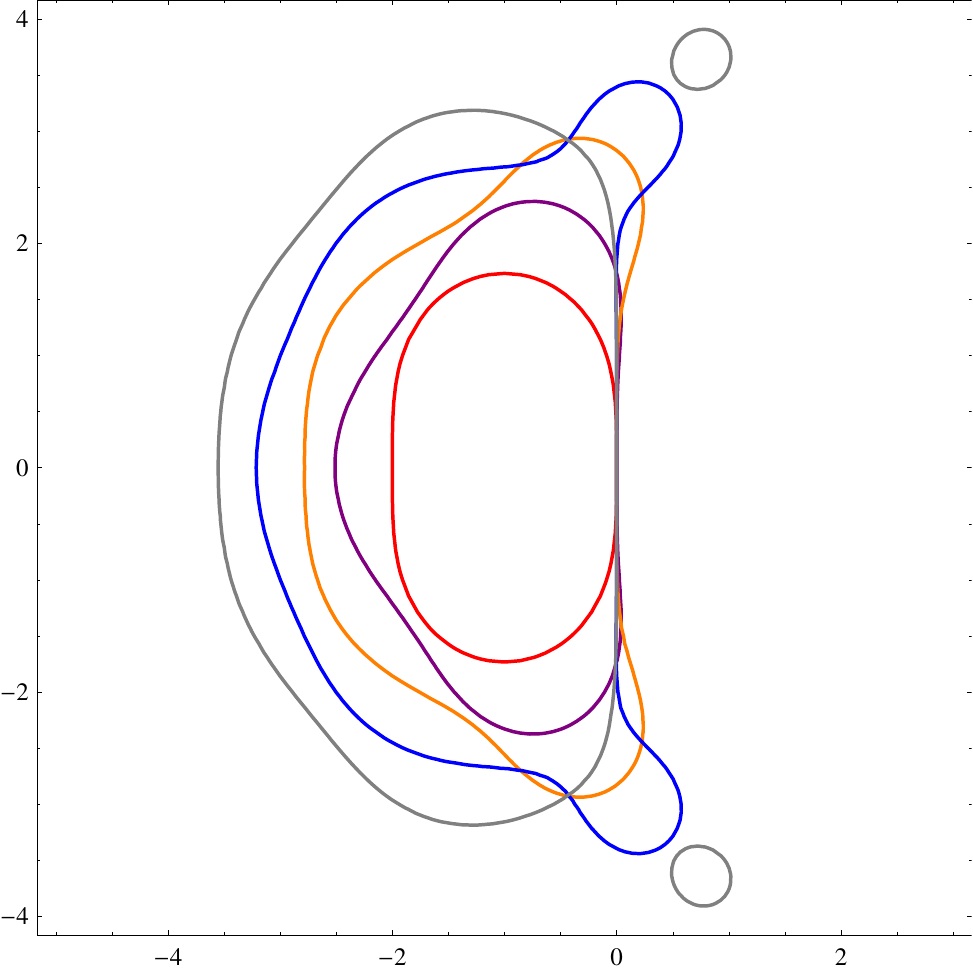}
        \subcaption{Contours of $\mathcal{R}^{r}_{ref}$.}
    \end{subfigure}
    \begin{subfigure}{0.09\textwidth}
    \end{subfigure}
    \begin{subfigure}{0.45\textwidth}
    \centering
        \includegraphics[width=\linewidth]{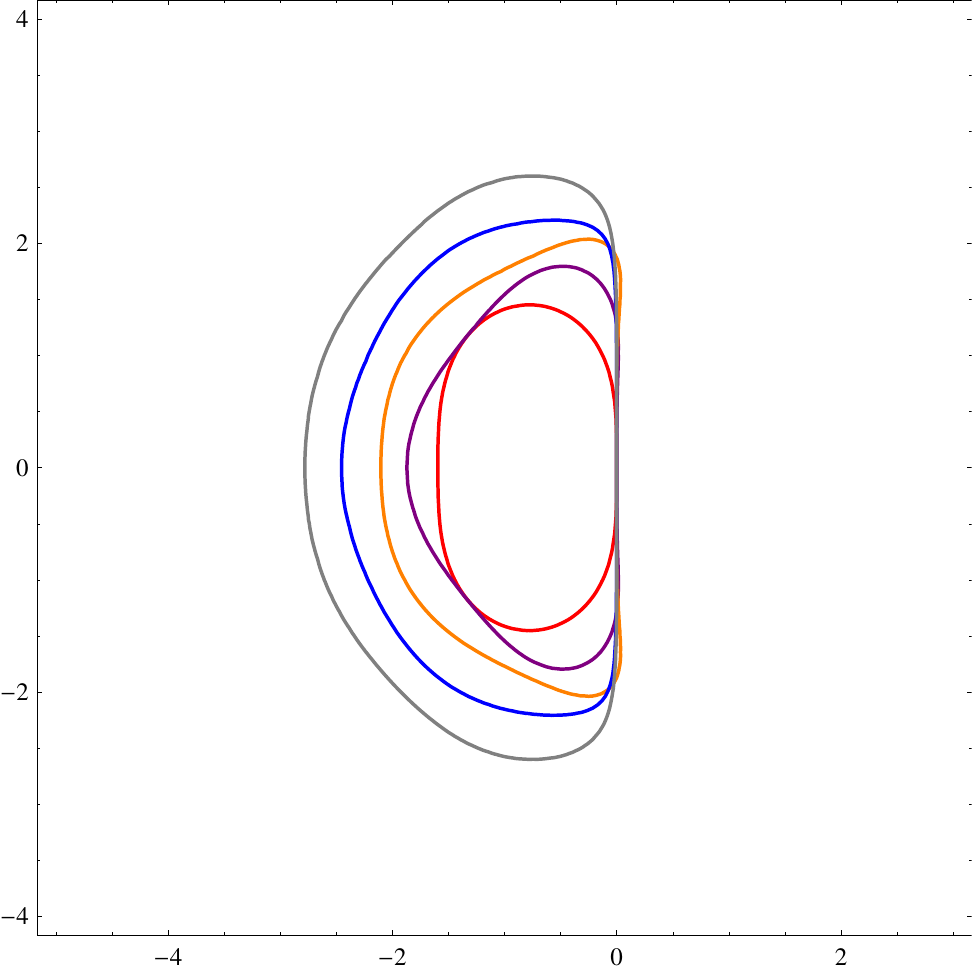}
        \subcaption{Contours of $\mathcal{R}^{r}_{MS}$.}
    \end{subfigure}
    \begin{subfigure}{0.45\textwidth}
    \centering
        \includegraphics[width=\linewidth]{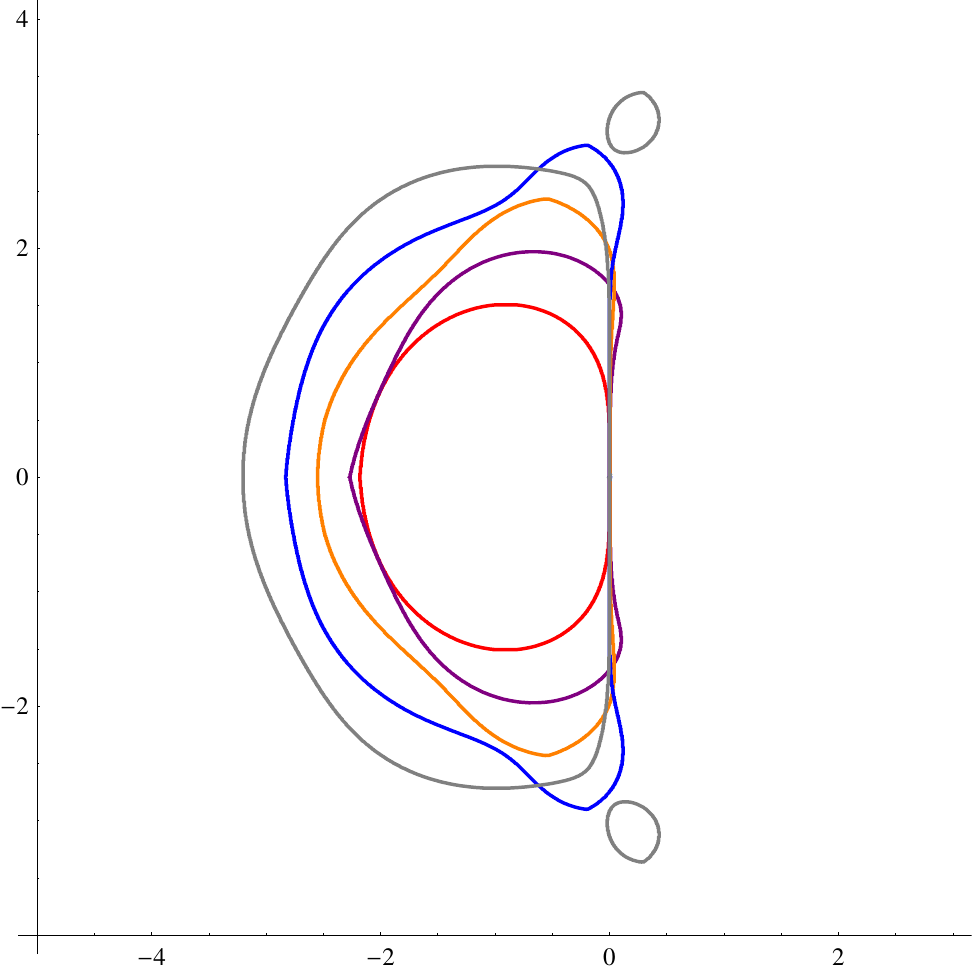}
        \subcaption{Contours of $\mathcal{R}^{r}_{AS}$.}
    \end{subfigure}
    \caption{Contours of reference sets, regions of mean-sqare stability, and regions of asymptotic stability of the randomized Taylor schemes for $r=0$ (red line), $r=1$ (purple line), $r=2$ (orange line), $r=3$ (blue line), and $r=4$ (gray line).}
    \label{fig:regions-ref-ms}
\end{figure}

\begin{remark} \label{remark:inclusions}
None of the following inclusions holds in general (for every $r\in\mathbb{N}$):
\begin{itemize}
\setlength{\itemsep}{4pt}
    \item[a)] $\mathcal{R}^{r}_{ref} \subset \mathcal{R}^{r+1}_{ref}$;
    \item[b)] $\mathcal{R}^{r}_{MS} \subset \mathcal{R}^{r+1}_{MS}$;
    \item[c)] $\mathcal{R}^{r}_{AS} \subset \mathcal{R}^{r+1}_{AS}$.
\end{itemize}

Let us consider $z_a=-0.6+2.8\,i$, $z_b=-0.03+1.9\,i$, and $z_c=-0.25+2.75\,i$. First two of these points are represented as the intersection of dashed lines in Figure \ref{fig:regions-inclusions}. We have
$$
    \Bigl| \sum_{j=0}^4 \frac{z_a^j}{j!} \Bigr|^2 = \frac{253\,409}{360\,000} < 1 \ \text{ and } \ \Bigl| \sum_{j=0}^5 \frac{z_a^j}{j!} \Bigr|^2 = \frac{5\,828\,357}{5\,625\,000} > 1.
$$
Thus, by \eqref{eq:ref}, $z_a\in \mathcal{R}_{ref}^2\setminus \mathcal{R}_{ref}^3$. Furthermore,
\begin{align*}
    F_2(z_b) & =\frac{2\,460\,549\,996\,776\,228\,711}{2\,520\,000\,000\,000\,000\,000}<1, \\[2pt] F_3(z_b) & =\frac{531\,703\,423\,127\,449\,318\,399\,669}{518\,400\,000\,000\,000\,000\,000\,000}>1,
\end{align*} 
which means that $z_b\in \mathcal{R}_{MS}^2\setminus \mathcal{R}_{MS}^3$, cf. \eqref{eq:MS}. Finally, using the \textit{scipy.integrate.quad} function in Python, we obtain the following estimates:
$$ G_3(z_c)\approx -0.41731 <0 \ \ \text{ and } \ \ G_4(z_c)\approx 0.06505 > 0. $$
Thus, by \eqref{eq:RAS-SP}, $z_c\in \mathcal{R}^{3}_{AS} \setminus \mathcal{R}^{4}_{AS}$.

\begin{figure}
    \centering
    \begin{subfigure}{0.45\textwidth}
    \centering
        \includegraphics[width=\linewidth]{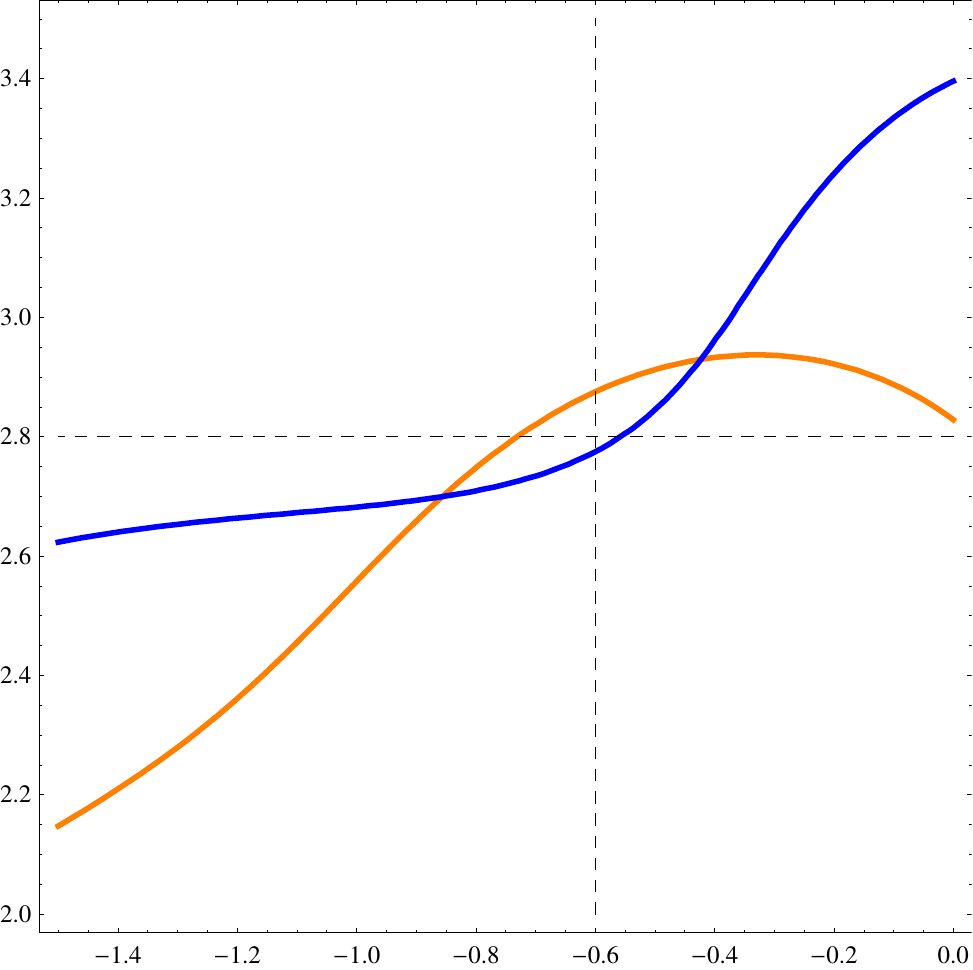}
        \subcaption{Contours of $\mathcal{R}^{2}_{ref}$ and $\mathcal{R}^{3}_{ref}$.}
    \end{subfigure}
    \begin{subfigure}{0.09\textwidth}
    \end{subfigure}
    \begin{subfigure}{0.45\textwidth}
    \centering
        \includegraphics[width=\linewidth]{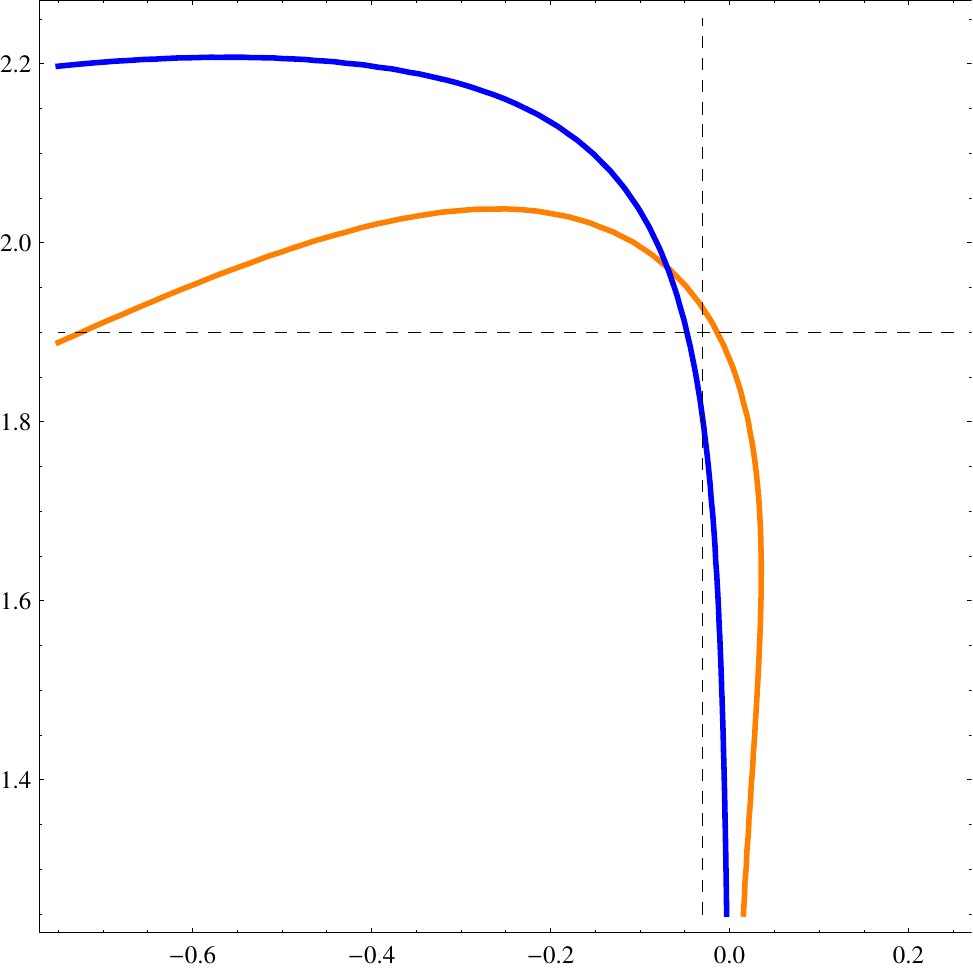}
        \subcaption{Contours of $\mathcal{R}^{2}_{MS}$ and $\mathcal{R}^{3}_{MS}$.}
    \end{subfigure}
    \caption{Fragments of contours of reference sets and regions of mean-sqare stability of the randomized Taylor schemes for $r=2$ (orange line) and $r=3$ (blue line).}
    \label{fig:regions-inclusions}
\end{figure}
\end{remark}

\begin{remark} \label{remark:as-ref}
In general, there is no inclusion between $\mathcal{R}^{r}_{AS}$ and $\mathcal{R}^{r}_{ref}$. 

To see this, let us consider $r=0$ and take $z_1=-2.1$, $z_2=-1+1.6\,i$. Then
\begin{align*}
\Bigl|1+z_1+\frac{z_1^2}{2}\Bigr| = 1.105  > 1 \ \text{ and } \ \Bigl|1+z_2+\frac{z_2^2}{2}\Bigr| = 0.78 < 1.
\end{align*}
On the other hand, we have estimated $G_0(z_1)\approx -0.07784$ and $G_0(z_2)\approx 0.13565$. By \eqref{eq:G}, \eqref{eq:RAS-SP} and \eqref{eq:ref}, we obtain $z_1\in \mathcal{R}^{0}_{AS}\setminus \mathcal{R}^{0}_{ref}$ and $z_2\in \mathcal{R}^{0}_{ref}\setminus \mathcal{R}^{0}_{AS}$.
\end{remark}

\begin{remark} \label{remark:Cminus}
In general, regions $\mathcal{R}^{r}_{ref}$, $\mathcal{R}^{r}_{MS}$ and $\mathcal{R}^{r}_{AS}$ are not included in $\mathbb{C}_-$. 

This inclusion is true for $r=0$, cf. Theorem 3(ii) and Theorem 4(iii) in \cite{randRK}, but does not hold for $r=1$, as shown in the following example:
$$F_1 (0.01+i) = \frac{19\,772\,000\,147\,001}{20\,000\,000\,000\,000}<1,$$
which combined with \eqref{eq:MS} and \eqref{eq:MSbound} implies that $0.01+i\in \mathcal{R}^1_{MS}\cap\mathcal{R}^1_{ref}\cap\mathcal{R}^1_{AS}\cap\mathbb{C}_+$.
\end{remark}

\begin{remark} \label{remark:convex}
In general, sets $\mathcal{R}^{r}_{ref}$, $\mathcal{R}^{r}_{MS}$ and $\mathcal{R}^{r}_{AS}$ are not convex.

Let $z_1=0.01+i$ and $z_2=0.01-i$. We know that $z_1,z_2\in \mathcal{R}^1_{MS}\cap\mathcal{R}^1_{ref}\cap\mathcal{R}^1_{AS}$, cf. Remark \ref{remark:Cminus} and Theorem \ref{theorem_MS_1}. On the other hand, $(\mathcal{R}^1_{MS}\cup\mathcal{R}^1_{ref} \cup\mathcal{R}^1_{AS})\cap\mathbb{R} \subset (-\infty,0)$ because $f_{r,z}(t) \geq 1$ for all $r\in\mathbb{N}$, $z\in [0,\infty)$, $t\in [0,1]$, and we have \eqref{eq:G}, \eqref{eq:RAS-SP}, \eqref{eq:ref}, \eqref{eq:MSbound}. Thus, $\frac{z_1+z_2}{2} = 0.01 \notin \mathcal{R}^1_{MS}\cup\mathcal{R}^1_{ref}\cup\mathcal{R}^1_{AS}$.
\end{remark}

\begin{remark} \label{remark:connect}
In general, sets $\mathcal{R}^{r}_{ref}$ and $\mathcal{R}^{r}_{AS}$ are not connected. 

Let us take $z_1 = 0.75+3.5\,i$ and $z_2=-0.25+2.5\,i$. We have 
\begin{align*}
    \Bigl| \sum_{j=0}^6 \frac{z_1^j}{j!} \Bigr|^2 & =\frac{27\,473\,196\,877\,335\,817\,540\,321}{121\,029\,087\,867\,608\,368\,152\,576}<1, \\[2pt] \Bigl| \sum_{j=0}^6 \frac{z_2^j}{j!} \Bigr|^2  & = \frac{48\,715\,333\,577\,673\,689\,545\,536\,241}{75\,643\,179\,917\,255\,230\,095\,360\,000}<1
\end{align*}
but 
$$\Bigl| \sum_{j=0}^6 \frac{\Bigl( \frac{z_1+z_2}{2} \Bigr)^j}{j!} \Bigr|^2   = \frac{9\,427\,129\,581\,150\,440\,422\,815\,049}{3\,025\,727\,196\,690\,209\,203\,814\,400}>1.$$
Thus, $z_3,z_4\in \mathcal{R}^4_{ref}$ but $\frac{z_3+z_4}{2}\notin \mathcal{R}^4_{ref}$, cf. \eqref{eq:ref}, which means that the set $\mathcal{R}^4_{ref}$ is disconnected.

For $z_3 = -0.5 + 2\,i$ and $z_4=0.25+3.25\,i$ we get 
$$G_4(z_3)\approx -0.50028<0, \ \ G_4(z_4)\approx -0.47024<0, \ \text{ and } \ G_4\Bigl(\frac{z_3+z_4}{2}\Bigr)\approx 0.03656>0,$$
which implies that $z_3,z_4\in\mathcal{R}^{4}_{AS}$ but $\frac{z_3+z_4}{2}\notin\mathcal{R}^{4}_{AS}$, see \eqref{eq:RAS-SP}. Hence, the set $\mathcal{R}^{4}_{AS}$ is disconnected.

In general, disconnectivity of the stability region would indicate that the method's behaviour for stiff problems is in a sense unpredictable -- taking smaller $h$ would not necessarily improve the method's performance. However, disconnected parts of the reference and asymptotic stability regions are observed in the right complex half-plane, which is out of interest in the context of A-stability. We conjecture that the intersection of each of the aforementioned stability regions with the left half-plane is connected.
\end{remark}

\section{Conclusions} \label{sec:concl}

We have established fundamental properties of probabilistic stability regions for randomized Taylor schemes. In particular, we have shown that notions of asymptotic stability and stability in probability are equivalent for this family of schemes, cf. \eqref{eq:RAS-SP}. Furthermore, we have proven openness and symmetry of all considered stability regions (see Theorem \ref{theorem_MS_1}), as well as their boundedness, cf. \eqref{eq:MSbound} in Theorem \ref{theorem_MS_2}. 

Although randomized Taylor schemes are not A-stable for any $r\in\mathbb{N}$ and in any probabilistic sense, the union of (asymptotic or mean-square) stability regions over all $r\in\mathbb{N}$ covers the entire left complex half-plane, see \eqref{eq:C-} in Theorem \ref{theorem_MS_2}. Hence, if the right-hand side function $f$ is sufficiently regular (in the most optimistic scenario, analytical), one may increase $r$ in order to prevent rapid variation in the approximated solution of a stiff problem. Otherwise, applicability of the methods studied in this paper is in practice limited to non-stiff problems.

Finally, we have ruled out a number of hypotheses concerning the potential properties of stability regions for randomized Taylor schemes, including their monotonicity (with respect to $r$) and some other potential inclusions, convexity, and connectivity, cf. Remarks \ref{remark:inclusions}--\ref{remark:connect}.

\section*{Acknowledgments}

I would like to thank Professor Paweł Przybyłowicz for many inspiring discussions while preparing this manuscript. 

I am also very grateful to the anonymous reviewer for valuable comments and suggestions that allowed me to improve the quality of the paper.

This research was funded in whole or in part by the National Science Centre, Poland, under project 2021/41/N/ST1/00135. For the purpose of Open Access, the author has applied a~CC-BY public copyright licence to any Author Accepted Manuscript (AAM) version arising from this submission. 

\appendix
\section{Auxiliary lemmas} \label{sec:appendix}

In this section, we prove technical lemmas which are necessary to establish equality \eqref{eq:RAS-SP}, Theorem \ref{theorem_MS_1}, and Theorem \ref{theorem_MS_2}.

\begin{lemma} \label{lemma:R_AS}
For each $z\in\mathbb{C}$ and each $r\in\mathbb{N}$, the random variable $\ln |f_{r,z}(\tau)|$ is square-integrable.
\end{lemma}

\begin{proof}
We will show more, i.e.,
\begin{equation} \label{eq:finiteint}
    \int\limits_0^1  \ln^2 \bigl|z_1t^k+z_2\bigl|^2\,\mathrm{d}t < \infty
\end{equation}
for all $k\in\mathbb{Z}_+$ and $z_1,z_2\in\mathbb{C}$ such that $(z_1,z_2)\neq (0,0)$. Since the case $z_1=0$ is trivial, in the following we consider any $z_1, z_2 \in\mathbb{C}$ such that $z_1\neq 0$. If $-\frac{z_2}{z_1}\notin [0,1]$, then $z_1t^k+z_2 \neq 0$ for all $t\in [0,1]$ and \eqref{eq:finiteint} follows because the integrand is a continuous function of $t$ on the interval $[0,1]$. From this point we assume that $\alpha = -\frac{z_2}{z_1} \in [0,1]$. Then 
\begin{equation*} 
    \int\limits_0^1  \ln^2 \bigl|z_1t^k+z_2\bigl|\,\mathrm{d}t  = \int\limits_0^1 \Bigl( \ln |z_1| + \ln \bigl|t^k-\alpha\bigl| \Bigr)^2\,\mathrm{d}t.
\end{equation*}

If $\alpha=0$, both integrals
$$\int\limits_0^1 \ln \bigl|t^k-\alpha\bigl| \,\mathrm{d}t = k\int\limits_0^1 \ln t  \,\mathrm{d}t \ \text{ and } \ \int\limits_0^1 \bigl(\ln \bigl|t^k-\alpha\bigl|\bigr)^2 \,\mathrm{d}t = k^2\int\limits_0^1 (\ln t )^2 \,\mathrm{d}t$$
are finite. In case of $\alpha\in (0,1]$, there exists $\beta\in(0,1]$ such that $\alpha=\beta^k$ and we can write that
$$ \int\limits_0^1 \ln \bigl|t^k-\alpha\bigl| \,\mathrm{d}t = \int\limits_0^1 \ln |t-\beta | \,\mathrm{d}t + \int\limits_0^1 \ln \Bigl( \beta^{k-1} + \sum_{j=0}^{k-2} t^{k-1-j} \beta^j \Bigr) \,\mathrm{d}t.$$
The first integral above has one singularity but this is a well-known fact that it is finite. For the second one, we note that the integrand is a continuous function for $t\in [0,1]$. Furthermore, since $(a+b)^2\leq 2a^2+2b^2$ for any real numbers $a,b$, we obtain
$$\int\limits_0^1 \bigl( \ln \bigl|t^k-\alpha\bigl|  \bigr)^2\,\mathrm{d}t \leq 2\int\limits_0^1 \bigl(\ln |t-\beta |\bigr)^2 \,\mathrm{d}t + 2\int\limits_0^1 \Bigl[ \ln \Bigl( \beta^{k-1} + \sum_{j=0}^{k-2} t^{k-1-j} \beta^j \Bigr)\Bigr]^2\,\mathrm{d}t$$
and we may use similar arguments as before to justify that the above integrals are finite.
\end{proof}

\begin{lemma}
\label{lemma_F}
For each $r\in\mathbb{N}$, the function $F_{r}$ is continuous in $\mathbb{C}$.
\end{lemma}

\begin{proof}
Let us fix $r\in\mathbb{N}$ and consider any $z, h\in\mathbb{C}$. Then
\begin{align}
    \bigl| F_r(z+h) - F_r(z) \bigr| & \leq \mathbb{E} \Bigl| |f_{r,z+h}(\tau)|^2 - |f_{r,z}(\tau)|^2 \Bigl| \notag \\ & \leq \mathbb{E} \Bigl( |f_{r,z+h}(\tau) - f_{r,z}(\tau)| \cdot \bigr( |f_{r,z+h}(\tau)| + |f_{r,z}(\tau)| \bigl) \Bigr). \label{eq:F1}
\end{align}
We note that
\begin{align*}
    |f_{r,z+h}(\tau)| & \leq \sum_{j=0}^{r+1} \frac{|z+h|^j}{j!} + \frac{|z+h|^{r+2}}{(r+1)!}  \leq \sum_{j=0}^{r+1} \frac{(|z|+|h|)^j}{j!} + \frac{(|z|+|h|)^{r+2}}{(r+1)!} =: \alpha_{z}(h)
\end{align*}
with probability $1$. Inserting this bound into \eqref{eq:F1} yields
\begin{align*}
    \bigl| F_r(z+h) - & F_r(z) \bigr| \leq  \bigl( \alpha_z(h) + \alpha_z(0) \bigr) \cdot \mathbb{E} |f_{r,z+h}(\tau) - f_{r,z}(\tau)| 
    \\ & = |h| \cdot \bigl( \alpha_z(h) + \alpha_z(0) \bigr) \cdot \mathbb{E}\Bigl| \sum_{j=1}^{r+1}\sum_{k=0}^{j-1} \frac{(z+h)^{j-1-k}z^{k}}{j!} + \tau^{r+1}\sum_{k=0}^{r+1} \frac{(z+h)^{r+1-k}z^k}{(r+1)!} \Bigr| \\ & \leq |h| \cdot \bigl( \alpha_z(h) + \alpha_z(0) \bigr) \cdot  \Bigl( \sum_{j=1}^{r+1} \frac{(|z|+|h|)^{j-1}}{(j-1)!} + (r+2)\cdot \frac{(|z|+|h|)^{r+1}}{(r+1)!} \Bigr).
\end{align*}
The last expression tends to $0$ when $h\to 0$, which completes the proof.
\end{proof}

\begin{lemma} \label{lemma:aux}
Let $k\in\mathbb{Z}_+$ and $\alpha\in [0,1]$. Then for each $\varepsilon >0$ there exists $\delta >0$ such that for any $z\in\mathbb{C}$ with $|z-\alpha|<\delta$ we have
\begin{equation*}
\int\limits_{A_\delta} \bigl| \ln |t^k-z| \bigr|\,\mathrm{d}t  < \varepsilon,
\end{equation*}
where $A_\delta = \{ t\in [0,1] \colon |t^k-\alpha|\leq \delta\}$.
\end{lemma}
\begin{proof}
Let $k\in\mathbb{Z}_+$, $\alpha\in [0,1]$ and $\varepsilon>0$ be fixed. Let us take into consideration only $\delta\in(0,\frac12)$. Then $|t^k-z| \leq |t^k-\alpha|+|z-\alpha| <2\delta <1$ for all $z\in\mathbb{C}$ such that $|z-\alpha|<\delta$ and all $t\in A_\delta$. As a result,
\begin{equation} \label{eq:step1}
	\int\limits_{A_\delta} \bigl| \ln |t^k-z|\bigr|\,\mathrm{d}t = -\int\limits_{A_\delta} \ln |t^k-z|\,\mathrm{d}t.
\end{equation}
Moreover,
\begin{equation} \label{eq:step2}
	0\geq \int\limits_{A_\delta} \ln |t^k-z|\,\mathrm{d}t \geq \int\limits_{A_\delta} \ln |t^k-\Re(z)|\,\mathrm{d}t \geq \inf_{\substack{x\in\mathbb{R} \\ |x-\alpha|<\delta}} \ \ \int\limits_{A_\delta} \ln |t^k-x|\,\mathrm{d}t
\end{equation}
because $|t^k-z| =\sqrt{\Re^2(t^k-z)+\Im^2(t^k-z)}\geq |\Re(t^k-z)| = |t^k-\Re(z)|$ and similarly $|\alpha-\Re(z)|\leq |\alpha-z|<\delta$. For $x\leq 0$ we obtain 
\begin{equation} \label{eq:step4}
    \int\limits_{A_\delta} \ln |t^k-x|\,\mathrm{d}t \geq k\int\limits_{A_\delta} \ln t \,\mathrm{d}t > -\varepsilon
\end{equation}
if $\delta$ is sufficiently close to $0$. Now let us consider $x\in (0,\alpha+\delta)$ and define $y=\sqrt[k]{x}$,
$$d(\delta) = \left\{ \begin{array}{ll} 0, & \text{when } \alpha-\delta\leq 0, \\ \sqrt[k]{\alpha-\delta}, & \text{when } \alpha-\delta >0. \end{array} \right.$$
Then
\begin{align} 
    \int\limits_{A_\delta} \ln |t^k-x|\,\mathrm{d}t & = \int\limits_{d(\delta)}^{\sqrt[k]{\alpha+\delta}} \ln |t-y|\,\mathrm{d}t + \int\limits_{d(\delta)}^{\sqrt[k]{\alpha+\delta}} \ln \Bigl( \sum_{j=0}^{k-1} t^{k-1-j}y^j \Bigr)\,\mathrm{d}t \notag \\ & \geq \int\limits_{d(\delta)-y}^{\sqrt[k]{\alpha+\delta}-y} \ln |t|\,\mathrm{d}t + (k-1)\int\limits_{d(\delta)}^{\sqrt[k]{\alpha+\delta}} \ln t\,\mathrm{d}t \notag \\  & \geq \ \ k \!\!\!\! \int\limits_{-\frac12(\sqrt[k]{\alpha+\delta}-d(\delta))}^{\frac12(\sqrt[k]{\alpha+\delta}-d(\delta))} \!\!\!\!\!\!\!\!\!\! \ln |t|\,\mathrm{d}t > -\varepsilon \label{eq:step5}
\end{align}
for $\delta$ sufficiently close to $0$. In the last two lines of \eqref{eq:step5} we used the fact that
$$\int\limits_{a}^{b} \ln |t|\,\mathrm{d}t \geq \int\limits_{-\frac{b-a}{2}}^{\frac{b-a}{2}} \ln |t|\,\mathrm{d}t $$
for all $a,b\in\mathbb{R}$ such that $a<b$; moreover, the last integral tends to $0$ when $\frac{b-a}{2}\to 0$. By \eqref{eq:step1}, \eqref{eq:step2}, \eqref{eq:step4} and \eqref{eq:step5} we get the desired claim.
\end{proof}

The following Lemma \ref{A-cont} is a generalization of Proposition 1 from \cite{randRK}.

\begin{lemma} \label{A-cont}
For each $r\in\mathbb{N}$, the function $G_r$ is continuous in $\mathbb{C}$.
\end{lemma}

\begin{proof}[Proof of Lemma \ref{A-cont}]
To see that $G_r$ is continuous in $0$, let us observe that 
$$ 0 <  1-\sum_{j=1}^{r+1} \frac{|z|^j}{j!} - \frac{|z|^{r+2}}{(r+1)!}  \leq \bigl| f_{r,z}(t) \bigr| \leq  1 + \sum_{j=1}^{r+1} \frac{|z|^j}{j!} + \frac{|z|^{r+2}}{(r+1)!}$$ for $|z|$ sufficiently close to $0$ and for all $t\in [0,1]$. Hence,
$$  \ln \Bigl( 1-\sum_{j=1}^{r+1} \frac{|z|^j}{j!} - \frac{|z|^{r+2}}{(r+1)!} \Bigr) \leq  G_r(z) \leq  \ln \Bigl( 1+ \sum_{j=1}^{r+1} \frac{|z|^j}{j!} + \frac{|z|^{r+2}}{(r+1)!}  \Bigr), $$
which implies that $\displaystyle \lim_{z\to 0} G_r(z) = 0 = G_r(0)$.

Let us consider $z\in\mathbb{C}\setminus\{0\}$. Then $G_r(z)$ can be expressed as
\begin{equation} \label{eq:G2}
    G_r(z)=\ln\Bigl(\frac{|z|^{r+2}}{(r+1)!} \Bigr) + H_r(h_r(z)),
\end{equation}
where 
\begin{align}
    h_r\colon & \mathbb{C}\setminus\{0\}\ni z \mapsto -\sum_{j=0}^{r+1} \frac{(r+1)!}{j!\cdot z^{r+2-j}} \in\mathbb{C}, \label{eq:h} \\
    H_r\colon & \mathbb{C} \ni z \mapsto \int\limits_0^1 \ln |t^{r+1}-z|\,\mathrm{d}t\in\mathbb{R}. \label{eq:H}
\end{align}
To complete the proof, it suffices to show that $H_r$ is continuous in $\mathbb{C}$. 

Firstly, let us consider a fixed $z\in\mathbb{C}\setminus [0,1]$. Let us define $$\delta = \min_{t\in [0,1]} |z-t^{r+1}|$$ and take any $(z_n)_{n=1}^\infty \subset \mathbb{C}$ such that $z_n \to z$, $z_n\neq z$ and $|z_n-z|<\frac{\delta}{2}$ for $n\in\mathbb{Z}_+$. Then for all $t\in [0,1]$ we have $|z_n-t^{r+1}|\geq |z-t^{r+1}|-|z_n-z|>\frac{\delta}{2}$. By Lagrange's mean value theorem and the triangle inequality, we obtain
$$\left|\frac{\ln |t^{r+1}-z|-\ln |t^{r+1}-z_n|}{z-z_n}\right|\leq \left|\frac{\ln |t^{r+1}-z|-\ln |t^{r+1}-z_n|}{|t^{r+1}-z|- |t^{r+1}-z_n|}\right|=\frac{1}{\xi(t,n)}$$ for some $\xi(t,n)$ falling between $|t^{r+1}-z|$ and $|t^{r+1}-z_n|$, provided that $|t^{r+1}-z|\neq |t^{r+1}-z_n|$. Since both these numbers are greater than $\frac{\delta}{2}$, we have $\xi(t,n)>\frac{\delta}{2}$ and thus 
\begin{equation} \label{eq:uniform_conv}
\bigl|\ln |t^{r+1}-z|-\ln |t^{r+1}-z_n|\bigr| < \frac{2}{\delta} \cdot |z-z_n|
\end{equation} 
for all $t\in [0,1]$. Note that the above inequality holds also when $|t^{r+1}-z| = |t^{r+1}-z_n|$. From \eqref{eq:uniform_conv} it follows that $t\mapsto \ln|t^{r+1}-z_n| $ converges uniformly to $t\mapsto \ln|t^{r+1}-z| $ for $t\in [0,1]$. Hence, 
\begin{equation*}
\lim_{n\to\infty}  \int\limits_0^1 \ln |t^{r+1}-z_n| \,\mathrm{d}t =  \int\limits_0^1  \ln |t^{r+1}-z| \,\mathrm{d}t
\end{equation*} 
and continuity of $G_r$ in each point $z\in\mathbb{C}\setminus [0,1]$ is proven.

Now let us consider a fixed $z\in [0,1]$ and set $\varepsilon >0$. By Lemma \ref{lemma:aux}, there exists $\delta >0$ such that for any $\zeta\in\mathbb{C}$ with $|\zeta-z|<\delta$ we have
\begin{equation} \label{eq1}
\int\limits_{A_\delta} \bigl| \ln |t^{r+1}-\zeta| \bigr|\,\mathrm{d}t  < \frac{\varepsilon}{3},
\end{equation}
where $A_\delta = \{ t\in [0,1] \colon |t^{r+1}-z|\leq \delta\}$. Let us consider any $(z_n)_{n=1}^\infty \subset \mathbb{C}$ such that $z_n\to z$, $z_n\neq z$ and $|z_n-z|<\frac{\delta}{2}$ for all $n\in\mathbb{Z}_+$. For each $t\in [0,1]\setminus A_\delta$ we have $|t^{r+1}-z|>\delta$ and $|t^{r+1}-z_n|>\frac{\delta}{2}$, $n\in\mathbb{Z}_+$. Thus, we can prove the uniform convergence of $t\mapsto \ln|t^{r+1}-z_n| $ to $t\mapsto \ln|t^{r+1}-z| $ for $t\in [0,1]\setminus A_\delta$ in a similar fashion as in the case of $z\in\mathbb{C}\setminus [0,1]$, see \eqref{eq:uniform_conv}. Hence, for sufficiently big $n$ we obtain
\begin{equation} \label{eq4}
I = \Bigl| \ \int\limits_{[0,1]\setminus A_\delta}   \ln |t^{r+1}-z_n| \,\mathrm{d}t - \int\limits_{[0,1]\setminus A_\delta}   \ln |t^{r+1}-z| \,\mathrm{d}t\Bigr|<\frac{\varepsilon}{3}.
\end{equation}
By \eqref{eq1} and \eqref{eq4}, 
$$ \Bigl| \int\limits_0^1  \ln |t^{r+1}-z_n| \,\mathrm{d}t -  \int\limits_0^1  \ln |t^{r+1}-z| \,\mathrm{d}t\Bigr| \leq I +\Bigl| \int\limits_{A} \ln |t^{r+1}-z_n| \,\mathrm{d}t \Bigr|+ \Bigl|\int\limits_{A} \ln |t^{r+1}-z| \,\mathrm{d}t\Bigr| < \varepsilon$$
for sufficiently big $n$. This concludes the proof.
\end{proof}

\end{document}